\documentclass{amsart}
\usepackage{amsmath,amsfonts,amssymb}
\usepackage{graphicx,psfrag,subfigure}
\usepackage{tikz-cd}
\usepackage[utf8]{inputenc}
\usepackage[T1]{fontenc}
\usepackage{lmodern}
\newtheorem{thm}{Theorem}
\newtheorem{rk}{Remark}
\newtheorem{prop}{Proposition}
\newtheorem{clly}{Corollary}
\newtheorem{lemma}{Lemma}
\newtheorem{defi}{Definition}

\newcommand{\N}{{\mathbb{N}}}

\begin{document}
\title{Non-locally connected spaces I:\\ Coverings and branched coverings on metric spaces.}
\author{Jorge Iglesias, Aldo Portela, Alvaro Rovella and Juliana Xavier}
 
 \maketitle

\begin{abstract}
Whether a rational function can have an indecomposable continuum as its Julia set is a fundamental open problem in complex dynamics. We settle this negatively for one of the most
classical examples of an indecomposable continuum: we prove that the Knaster continuum cannot be the Julia set of any rational function. To obtain this result, we develop a theory
of coverings and branched coverings on general (non-locally connected) metric spaces.  We introduce a new Euler characteristic and extend the Riemann--Hurwitz formula to a wider class of spaces, including continua
for which the classical shape-theoretic invariants are trivial.
\end{abstract}

\section{Introduction}

\subsection*{Motivation and main result}

Classically, a covering map $f:X\to Y$ is a continuous map such that every $y\in Y$ has an open neighborhood $V$ whose preimage is a disjoint union of open sets $U_i$, each
of which is mapped homeomorphically onto $V$. As explained in \cite{lub}, when the spaces are locally connected, $V$ can be assumed to be connected, so the sets $U_i$ are
precisely the connected components of $f^{-1}(V)$ and are therefore intrinsically determined. This is no longer true in the non-locally connected setting: there may be several
different partitions of $f^{-1}(V)$ satisfying the required property. To overcome this difficulty, we use the metric structure of the spaces directly, through the notion of
{\em chains} --- finite sequences of points, each sufficiently close to the following one --- as a substitute for curves.\\

Covering spaces of non-locally connected continua have been studied from several perspectives; see \cite{nadler} for background on indecomposable continua. Fox developed the
theory of {\em overlays}, extending classical covering space theory to arbitrary metrizable spaces and replacing the fundamental group by its shape-theoretic counterpart, the
fundamental pro-group \cite{fox}. Using Fox's theory of overlays, Boroński and Sturm proved that a $P$-adic pseudosolenoid $S$ admits a $d$-fold covering onto itself if and 
only if $d$ is relatively prime to all but finitely many primes in $P$, and that any such covering space is unique up to homeomorphism \cite{boronski-sturm}. Coverings of
individual continua by other continua have also been studied outside this framework: Dębski, Heath and Mioduszewski characterized exactly 2-to-1 maps from continua onto
tree-like continua such as the Knaster continuum \cite{dhm} -- though, it will follow from our work that such maps cannot be coverings of the Knaster continuum in our sense.
Our approach is complementary: rather than the shape-theoretic viewpoint of overlays, we work directly with the metric structure of the continuum through chains, and our methods 
extend naturally to branched coverings of general metric continua, which is essential for the dynamical application below.\\

Our motivation comes from the dynamics of branched coverings on the sphere, in particular of rational maps. A fundamental open problem in complex dynamics — posed in \cite{cmr} — 
asks whether a rational function can have an indecomposable continuum as its Julia set. Since Julia sets are completely invariant, such an example would give an 
indecomposable continuum carrying a branched self-covering; understanding the intrinsic dynamics of branched coverings on indecomposable continua can therefore constrain which 
continua may occur as invariant sets in ambient dynamics.\\

In this paper we settle this question negatively for one of the most classical examples of an indecomposable continuum. The Knaster continuum --- the classical "buckethandle" continuum, one of the standard first examples of an indecomposable continuum in continuum theory --- cannot be a completely invariant continuum for any branched covering of the sphere of degree greater than one. Consequently, the Knaster continuum cannot be the Julia set of a rational function. Precisely, we prove that every branched self-covering of the Knaster continuum satisfies a Riemann--Hurwitz relation incompatible with the classical Riemann--Hurwitz formula for branched coverings of the sphere. The obstruction is therefore not specific to holomorphic dynamics: the Knaster continuum is ruled out already at the topological level, for arbitrary branched coverings of the sphere.\\

This result fits into a broader question. Until recently it was not known whether there existed branched coverings of the sphere of degree $d$, $|d|>1$, with an indecomposable
completely invariant continuum. An example was constructed in \cite{plig}, but the resulting map is not Thurston equivalent to a rational map: it is a $(2,2,2,2)$ orbifold with
Anosov homology action, which is a Thurston obstruction.  Consequently, even in the $C^0$ setting, it remains unknown whether there is a Thurston map, Thurston equivalent to a rational map, with an indecomposable completely
invariant continuum — and likewise whether there is a plane branched covering of degree $d$, $|d|>1$, with an indecomposable completely invariant continuum at all. Our result
shows that the Knaster continuum, at least, cannot provide such an example.

\subsection*{New tools: chains and chain-simple-connectedness}

The purpose of this paper is to develop intrinsic tools for these questions, adapted to continua that need not be locally connected or path-connected. Chains characterize
connectedness and, under suitable hypotheses, can be uniquely lifted by coverings; this yields a covering theory for non-locally connected continua, together with an
alternative homotopy theory built on lifting continua rather than curves.\\

Within this framework we introduce a notion of simple connectedness that does not assume path-connectedness, which we call {\em chain-simply-connectedness}. We prove that
every connected covering of a chain-simply-connected continuum is a homeomorphism. The notion is genuinely different from classical simple connectedness: the quasi-circle $Y$
of Figure \ref{quasicircle}, although simply connected in the classical sense, is not chain-simply-connected — a nontrivial connected covering of $Y$ is given by the quasi-line
of Figure \ref{quasiline}. On the other hand, non-separating plane continua {\em are} chain-simply-connected.\\

\begin{figure}[h]
\centering
{\includegraphics[scale=0.4]{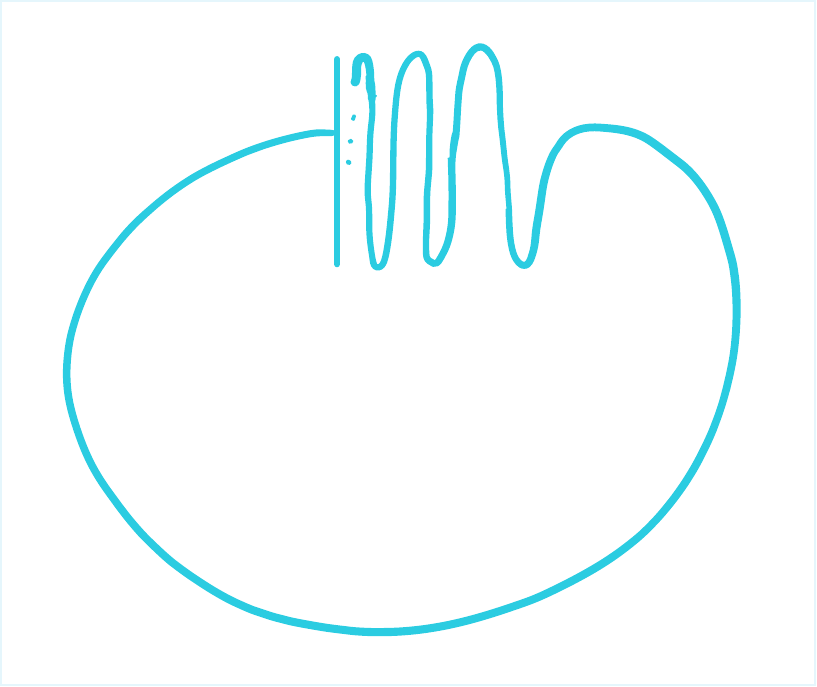}}
\caption{The quasi-circle.}
\label{quasicircle}
\end{figure}

\begin{figure}[h]
\centering
{\includegraphics[scale=0.7]{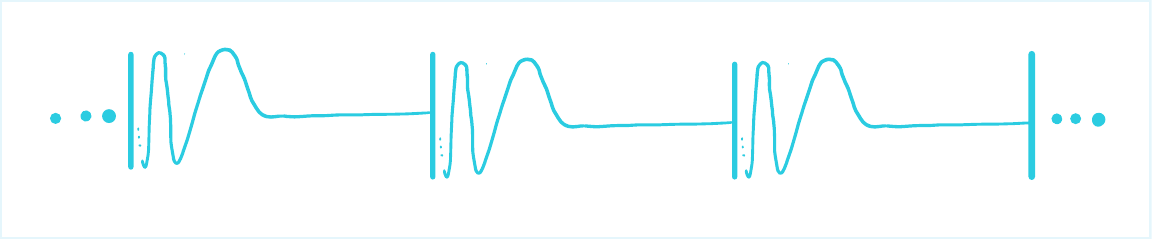}}
\caption{The quasi-line.}
\label{quasiline}
\end{figure}

It is known that the dyadic solenoid admits $d:1$ self-coverings for every odd $d$ and is therefore not chain-simply-connected.  This also provides examples of indecomposable 
continua admitting self-coverings; see \cite{fox,boronski-sturm} for the shape-theoretic and overlay perspective on coverings of solenoids and related continua. Such self-coverings 
are not easy to come by in general: for a compact connected space to admit a nontrivial covering of itself is a strong rigidity condition, ruling out most continua and imposing 
severe constraints even among those that do admit one -- as illustrated by the arithmetic condition on $d$ in \cite{boronski-sturm} above, by the general obstructions and 
constructions surveyed in \cite{timm-survey}, and by the obstruction we obtain for the Knaster continuum in this paper.\\

Beyond the covering case, we also allow critical points, i.e.\ we study branched coverings of general metric spaces and their lifting properties. This is essential for the dynamical
application, since completely invariant continua for rational maps --- and, more generally, for branched coverings of the sphere --- naturally interact with critical points. A 
completely invariant continuum for a branched covering of the sphere is shaped by two things at once: its own branched-covering structure, and the dynamics of the ambient 
sphere it sits in. Our obstruction for the Knaster continuum comes precisely from comparing the two: we show that a branched self-covering of the Knaster continuum satisfies 
a Riemann--Hurwitz relation that cannot occur for a branched covering of the sphere.\\

A Riemann--Hurwitz relation of a pair of spaces $(X,Y)$ is a formula relating the number of critical points to the degree of {\em any} branched covering $f:X\to Y$; the last
section of the paper gives the precise definition and explains in detail how the relation obtained for the Knaster continuum contradicts the classical Riemann--Hurwitz formula 
for sphere coverings. Along the way, we introduce a genuinely new Euler characteristic for continua such as the Knaster continuum, giving this notorious continuum a 
well-defined topological invariant it was not previously known to have: the shape-theoretic Euler characteristic is $1$ for the Knaster continuum (that of a point), 
while ours is $1/2$.

\subsection*{Organization of the paper}
Section \ref{2} introduces metric coverings, chains and the associated homotopy theory, with examples in Section \ref{4}. Chain-simply-connectedness is developed in
Section \ref{3}, where we prove:

\begin{thm}\label{t1}
Let $X$ be compact, $f:X\to Y$ a covering of degree $d$, and $K\subset Y$ a chain-simply-connected continuum. Then $f^{-1}(K)$ has $d$ connected components, each homeomorphic
to $K$ via $f$.
\end{thm}

As a corollary, every covering $f:X\to Y$ with $X$ connected and $Y$ a chain-simply-connected continuum must be a homeomorphism.

Branched coverings are introduced in Section \ref{6}, where the main result is:

\begin{thm}\label{t2}
Let $f$ be a branched covering between continua $X$ and $Y$. If $K\subset Y$ is compact and $K'$ is a component of $f^{-1}(K)$, then $f|_{K'}:K'\to K$ is a branched covering.
In particular, the number of connected components of $f^{-1}(K)$ is at most the degree of $f$.
\end{thm}

The exact number of components of $f^{-1}(K)$ depends on additional properties of the ambient spaces: if $K$ contains no critical values and is chain-simply-connected,
this number equals the degree of $f$ (Theorem \ref{t1}); for locally connected spaces it is related instead to the Euler characteristic.\\

Section \ref{7} also contains the paper's main application: we introduce a new Euler characteristic for spaces such as the Knaster continuum -- finer than the classical shape-theoretic one, which is trivial there -- derive the Riemann--Hurwitz relation satisfied by its branched self-coverings, and show that this relation is incompatible with the classical Riemann--Hurwitz formula for the sphere -- ruling out the Knaster continuum as the Julia set of a rational function.

\section{Metric coverings, chains and homotopy.}\label{2}

As explained in the introduction, we endow the space with a metric structure in order to avoid the ambiguities that appear in the standard definition of covering spaces when
the underlying space is not locally connected.\\

\begin{defi} 
Let $X$ be a metric space. We say that a continuous map $f:X \to Y$ is a {\em metric covering}, if for every point $y\in Y$ and any positive 
number $\epsilon$ 
there exists a neighborhood $U$ of $y$ such that:\\
\begin{enumerate}
\item
$f^{-1}(U)$ is a disjoint union of open sets each one of diameter less than $\epsilon$, and
\item
the restriction of $f$ to each of these open sets is a homeomorphism onto $U$. \\
\end{enumerate} 

The neighborhood $U$ of $y$ satisfying these conditions is called $\epsilon$-admissible.\\
\end{defi}

Note that a metric covering is a covering in the usual sense.\\

\begin{defi} 
Let $f:X\to Y$ be a covering map.\\

\begin{itemize}
\item
$f$ is {\emph {finite}} if the set $f^{-1}(y)$ is finite for all $y\in Y$.\\
\item
If $f$ is finite and $\# f^{-1}(y)=d$ for all $y\in Y$, we say that $d$ is the {\emph{degree}} of $f$ and we use 
the notation $d=\deg(f)$.\\
\item
$f$ is {\emph{ discrete}} if $X$ is a metric space and 
$$i(f):=\inf \{d(x, y): x\neq y, f(x)=f(y)\}>0.$$\\
\end{itemize}
\end{defi}
 
The following facts have elementary proofs:\\

\begin{lemma}
\label{finite} 
Let $f:X \to Y$ be a covering map between metric spaces.\\

\begin{enumerate}
\item
If $f$ is finite, and $X$ is metric, then $f$ is a metric covering.\\
\item
If $f$ is finite and $Y$ is connected, then $f$ has a well defined degree.\\
\item
If $X$ is compact, then $f$ is finite.\\
\end{enumerate}
\end{lemma}

The following assertion will be key in the rest of the article:\\

\begin{lemma}\label{leb} Let $f:X \to Y$ be a discrete metric covering with $Y$ a metric space, and $K\subset Y$ be compact. 
Then, for all $0<\epsilon \leq i(f)/2$, there exists $\delta >0$ such that if 
$d(x,y)<\delta$, $x,y \in K$ and $\tilde x\in f^{-1}(x)$, then there exists a unique $\tilde y\in f^{-1}(y)$ such that 
$d(\tilde x, \tilde y)<\epsilon$.
\end{lemma}
\begin{proof} Let $0<\epsilon \leq i(f)/2$. For all $y\in Y$ there exists $U(y)$ a neighborhood of $Y$ such that $f^{-1}(U)$ is a 
disjoint union of open sets each one of diameter less than $\epsilon$ and
the restriction of $f$ to each of these sets is a homeomorphism onto $U$.  
Let $\delta>0$ be the Lebesgue number of the covering of $K$ given by the $U(y), y\in K$.  Now, take 
$x,y\in K$ such that $d(x,y)<\delta$ and
$\tilde x\in f^{-1}(x)$.  We know that there exists $\tilde y \in f^{-1}(y)$ such that $d(\tilde y, \tilde x)<\epsilon$.  Now let $f(z)=y$ and
$z\in B(\tilde x, \epsilon)$. Then $d(z,\tilde y)\leq d(z,\tilde x)+d(\tilde x, \tilde y)< i(f)$, which gives $z=\tilde y$.\\
 
\end{proof}

\begin{defi}
\label{cadena} 
An $\varepsilon$- $n$- chain in $K$ is a finite sequence of points $x_0, x_1, \ldots, x_n$, where $x_i \in K$ and $d(x_i, x_{i+1}) < \varepsilon$.
We will denote the chain $x_0, \ldots, x_n$ by $\mathcal{A} = \{x_i\}_{i=0}^{n}$.\\  

Given two $\varepsilon$- $n$- chains in $K$, $\mathcal{A}_1  = \{x_i\}_{i=0}^{n}$ and $\mathcal{A}_2 = \{y_i\}_{i=0}^{n}$, we say that $d(\mathcal{A}_1, \mathcal{A}_2) < \delta$ if 
$d(x_i, y_i) < \delta$ for all $i = 0, \ldots, n$.\\

We say that the $\varepsilon$- $n$- chain $\mathcal{A} = \{x_i\}_{i=0}^{n}$ goes from $x$ to $y$ if $x_0=x$ and $x_n=y$.\\
\end{defi}
 
We say that $\mathcal{B}=\{y_i\}_{i=0}^m$ is a {\em refinement} of
$\mathcal{A}= \{x_i\}_{i=0}^n$ if $m>n$ and there exists $0\leq i_0<i_1<\dots<i_n = m$ such that $y_j=x_0$ for all $j\in \{0,\ldots, i_0\}$, $y_j=x_1$ for all $j\in \{i_0+1,\ldots, i_1\},
\ldots, y_j=x_n$ for all $j\in \{i_{n-1}+1,\ldots, i_n=m\}$.\\

The purpose of the definition of refinement is that one can refine a $\varepsilon$-$n$-chain to a $\varepsilon$-$m$-chain (with $m>n$) 
just repeating some points of the original chain.\\

Given a $\varepsilon$- $n$- chain in $K$, we will sometimes refer to it as {\it an $\varepsilon$- chain}, or simply {\it a chain} with no reference to 
$n$ and $\varepsilon$.\\

\begin{defi}\label{homotopia}
Given two $\varepsilon$-chains in \(K\), \(\mathcal{A}_1\) and \(\mathcal{A}_2\)  from \(x\) to \(y\), we say they are \(\delta\)-homotopic in $K$ 
if there exist
\(\mathcal{C}_1, \ldots, \mathcal{C}_m\)$, \varepsilon$- $n$- chains in $K$  from \(x\) to \(y\) 
such that \(\mathcal{C}_1\) is a refinement of \(\mathcal{A}_1\), 
\(\mathcal{C}_m\) is a refinement of \(\mathcal{A}_2\) , and 
\(d(\mathcal{C}_i, \mathcal{C}_{i+1}) < \delta\) for all $i=1,\ldots,m-1$.\\
\end{defi}

We now turn to lifting properties of discrete metric coverings. In what follows, $X$ and $Y$ will always be metric spaces.\\

\begin{defi} Let $f:X \to Y$ be a covering,
${\mathcal A}=\{y_i\}_{i=0}^n$ a chain
in $Y$ and $\varepsilon>0$. We say that the chain $\tilde {\mathcal A}=\{\tilde y_i\}_{i=0}^n$ $\varepsilon$-lifts ${\mathcal A}$, 
if $\tilde {\mathcal A}$ is a $\varepsilon$-chain in $X$ and $f(\tilde y_i)=y_i$ for all $i=0, \ldots, n$. \\
\end{defi}

\begin{lemma}
\label{lifts} 
Let $f:X \to Y$ be a discrete metric covering and $K\subset Y$ be compact. Then, for all $0<\varepsilon \leq i(f)/2$, there exists $\delta >0$ such that if 
${\mathcal A}=\{y_i\}_{i=0}^n$ is a $\delta$-chain
in $K$ and $\tilde y_0$ is a lift of $y_0$, then there exists a unique $\varepsilon$-lift of $\mathcal A$,  $\tilde {\mathcal A}=\{x_i\}_{i=0}^n$  
such that $x_0=\tilde y_0$. 
\end{lemma}
\begin{proof} Let $0<\varepsilon \leq i(f)/2$ and take $\delta >0$ as in Lemma \ref{leb}.  Now, if
${\mathcal A}=\{y_i\}_{i=0}^n$ is a $\delta$-chain
in $K$ and $\tilde y_0$ is a lift of $y_0$, then by   Lemma \ref{leb} there exists a unique lift $\tilde y_1$ of $y_1$ in $B(\tilde y_0, \varepsilon)$. 
We let $x_1=\tilde y_1$.  
Analogously, there exists a unique lift 
$\tilde y_2$ of $y_2$ in $B(x_1, \varepsilon)$ and we let $x_2=\tilde y_2$.  We carry on until we finish.\\
\end{proof}

\begin{lemma}
\label{endpts} 
Let $f:X \to Y$ be a discrete metric covering and $K\subset Y$ be compact. Let $\epsilon <i(f)/4$ and $\delta >0$ as in Lemma \ref{leb}.  Let
$\mathcal{A}= \{x_i\}_{i=0}^n$ and
$\mathcal{B}=\{y_i\}_{i=0}^n$ be $\delta$-chains in $K$ from $x$ to $y$ with $d(\mathcal{A},\mathcal{B})<\delta$. Fix $\tilde x\in f^{-1}(x)$ and let 
$\tilde {\mathcal{A}}= \{\tilde x_i\}_{i=0}^n$ and
$\tilde {\mathcal{B}}=\{\tilde y_i\}_{i=0}^n$ be the unique $\epsilon$-lifts of $\mathcal{A}$ and $\mathcal{B}$ starting at $\tilde x$. Then, 
$\tilde x_n = \tilde y_n$.\\
 \end{lemma}

\begin{proof}  We will prove that for all $i=1, \ldots,n $, $d(\tilde x_i, \tilde y_i)<\epsilon< i(f)$ which proves the lemma as both $\tilde x_n, \tilde y_n$ lift $y$.
For $i=1$, note that as $d(x_1, y_1)<\delta$ there exists a unique lift $y_1'$ of $y_1$ such that $d(\tilde x_1, y_1')<\epsilon$.  Then, 
$d(y_1', \tilde y_1)\leq d(y_1', \tilde x_1) + d(\tilde x_1, \tilde x) + d(\tilde x, \tilde y_1) <3\epsilon <i(f) $ which gives $y_1'=\tilde y_1$ and therefore 
$d(\tilde x_1, \tilde y_1)<\epsilon$.  Suppose now that $d(\tilde x_i, \tilde y_i)<\epsilon$ and let $y_{i+1}'$ be the lift of $y_{i+1}$ such that 
$d(\tilde x_{i+1}, y_{i+1}')<\epsilon$.  Then, $d(y_{i+1}', \tilde y_{i+1})\leq d(y_{i+1}', \tilde x_{i+1}) +
d(\tilde x_{i+1}, \tilde x_i) + d(\tilde x_i, \tilde y_i)+d(\tilde y_i,\tilde y_{i+1}) <4\epsilon <i(f) $ which gives $y_{i+1}'=\tilde y_{i+1}$ and therefore 
$d(\tilde x_{i+1}, \tilde y_{i+1})<\epsilon$. This finishes the proof.\\
 \end{proof}

As a corollary, we obtain:\\
 
\begin{clly}
\label{cl2} 
Let $f:X \to Y$ be a discrete metric covering, $K$ compact subset of $Y$ and $\epsilon<i(f)/4$.  Then, there exists $\delta >0$ such that:\\

\begin{enumerate}
 \item any $\delta$-chain in $K$ starting at $x$ has a unique $\epsilon$-lift starting at $\tilde x\in f^{-1}(x)$.\\
 
 \item  if $\mathcal{A}= \{x_i\}_{i=0}^n$ and
$\mathcal{B}=\{y_i\}_{i=0}^n$ are  $\delta$-homotopic $\delta$-chains in $K$ from $x$ to $y$, then the $\epsilon$-lifts of $\mathcal{A}$ and $\mathcal{B}$ starting at

$\tilde x$ have the same
endpoint.\\
\end{enumerate}
\end{clly}

\section{Chain-simply-connected spaces}\label{3}

In this section we define simple connectivity in terms of chains, for general metric spaces. Then the proof of Theorem \ref{t1} is given. 
In the last subsection conditions are given that imply chain simple connectivity for subsets of euclidean spaces.\\

We will say that a chain  $\mathcal{A} = \{x_i\}_{i=0}^{n}$  is closed if $x_0=x_n$ and is constant if $x_i = x$ for all $i = 0, \ldots, n$.\\

\begin{defi}
\label{simplementeconexo}
 A set \( K \) is chain-simply connected if for every \( \rho > 0 \) there exists \( 0< \varepsilon_0 < \rho \), such that for every
 \( 0<\varepsilon < \varepsilon_0 \) and for every \( n \in \mathbb{N} \), any closed \( \varepsilon -n \) chain in \( K \)  is \( \rho \)-homotopic to a
 constant chain. \\
\end{defi}

\begin{lemma}\label{alfa}  Let $f:X\to Y$ a discrete metric covering and $K\subset Y$ be compact and chain-simply connected.  Then, there exists $\alpha>0$ such that, 
for every $x\in K$ and $\tilde x\in f^{-1}(K)$:\\

\begin{enumerate}
 \item any $\alpha$-chain in $K$ starting at $x$ has a unique $\varepsilon$-lift starting at $\tilde x$\, where $\varepsilon=i(f)/4$.\\
 
 \item the $\varepsilon$-lift of any closed $\alpha$-chain
in $K$ is a closed chain .\\
\end{enumerate}
\end{lemma}
\begin{proof} As $K$ is compact,  we know from Corollary \ref{cl2} that there exists $\delta>0$ such that $\delta$-chains have unique lifts ant that the lifts of any two 
$\delta$- homotopic $\delta$-chains in $X$ that
start at the same point also finish
at the same point.  As $K$ is chain-simply connected there exists $0<\alpha<\delta$ such that any closed $\alpha$-chain in $K$  is $\delta$-homotopic
to the constant chain.  In particular, any $\alpha$-chain in $K$ lifts to closed chains in $X$.\\
\end{proof}
%

\begin{defi}  We say that $x$ and $y$ are chain-connected in $X$ if for all $\epsilon>0$ there exists an $\epsilon$-chain in $X$ from $x$ to $y$.  The chain-component in $X$ of 
$x\in X$ is the set of points in $X$ that are chain-connected to $x$.  We say that $X$ is chain-connected, 
if $x$ and $y$ are chain-connected in $X$ for all $x,y\in X$.\\
The chain components of $X$ are the maximal chain connected subsets of $X$.\\
\end{defi}

\begin{rk}  If $X$ is connected, then it is chain-connected.  The converse is true if $X$ is compact because chain components are closed.\\
 
\end{rk}

The fact that the quasi circle is not chain-simply-connected (as was announced in the Introduction) can be deduced from the following result:\\

\begin{lemma}\label{unifcont}  Let $f:X\to Y$ a discrete metric covering and $K\subset Y$ be compact and chain-simply connected. Suppose, furthermore that $f$ is  uniformly continuous
in $f^{-1}(K)$.
If $f(x_0)=f(x_1)=x\in K$, $x_0\neq x_1$, then $x_0$ and $x_1$ are not chain-connected in $f^{-1}(K)$.\\
 
\end{lemma}

\begin{proof}  We know from Lemma \ref{alfa} that closed $\alpha$-chains in $K$ lift to closed chains in $f^{-1}(K)$.  As $f$ is uniformly continuous in $f^{-1}(K)$, there exists
$\epsilon>0$ such that if $x,y\in f^{-1}(K)$, then $d(x,y)<\epsilon$ implies $d(f(x),f(y))<\alpha$. This means that $\epsilon$-chains in $f^{-1}(K)$ are mapped to 
$\alpha$-chains in $K$. The result follows immediately. \\
\end{proof}

Note that  using the metric covering from the quasi-line to the quasi-circle one deduces that the quasi-circle is not chain-simply-connected from the previous result.\\

%
%
%
%
%

\begin{lemma}\label{kalfa}  Let $f:X\to Y$ be a discrete metric covering and $K\subset Y$ be a continuum.  
Then $f^{-1}(K)= \cup_{\alpha} K_{\alpha}$ where each $K_{\alpha}$ is closed and
$f|_{K_{\alpha}}:K_{\alpha}\to K$ is a metric covering for all $\alpha$. \\
Moreover, if $x\in K$ and $f^{-1}(x)= \cup_{\alpha} \{x_{\alpha}\}$, then each $K_{\alpha}$ is the chain-component of $x_{\alpha}$ in $f^{-1}(K)$.\\

\end{lemma}

\begin{proof} Let $x\in K$ and $f^{-1}(x)= \cup_{\alpha} x_{\alpha}$. For all $\alpha$, let $K_{\alpha}$ be the chain-component of $f^{-1}(K)$ 
containing $x_\alpha$. Note that each $K_{\alpha}$ is closed. To show that $f|_{K_{\alpha}}:K_{\alpha}\to K$ is onto, we have to show that for all $y\in K$ and for all
$\epsilon >0$,
there exists an $\epsilon$-chain in $f^{-1}(K)$ from $x_{\alpha}$ to some $y_{\alpha}\in f^{-1}(y)$.\\

Take  $\epsilon < i(f)/2$ and $\delta >0$ as in Lemma \ref{lifts}. As $K$ is connected, for any $y\in K$
there exists a $\delta$-chain $\mathcal{A}$ in $K$ from $x$ to $y$.  Moreover, there is a unique $\epsilon$-chain $\tilde {\mathcal{A}}$  
starting at $x_{\alpha}$ and lifting $\mathcal{A}$.
Note that $\tilde {\mathcal{A}}$ is  an $\epsilon$-chain in $f^{-1}(K)$ from $x_{\alpha}$ to $y_{\alpha}\in f^{-1}(y)$ as wanted.\\

The remaining assertions follow immediately.\\
\end{proof}

%
%

%
%

\begin{clly}  Let $f:X\to Y$ be a  discrete metric covering and $K\subset Y$ be a chain-simply-connected continuum. 
Suppose that $f$ is uniformly continuous in $f^{-1}(K)$.
Then, $f^{-1}(K)= \cup_{\alpha} K_{\alpha}$ where
each $K_{\alpha}$ is closed, the 
union is disjoint, and
$f|_{K_{\alpha}}:K_{\alpha}\to K$ is a homeomorphism.   \\
\end{clly}

\begin{proof} From Lemma \ref{kalfa} we know that $f^{-1}(K)= \cup_{\alpha} K_{\alpha}$ where each $K_{\alpha}$ is closed,
$f|_{K_{\alpha}}:K_{\alpha}\to K$ is onto
for 
all $\alpha$, and each $K_{\alpha}$ is the chain-component of $x_{\alpha}$ in $f^{-1}(K)$, where $f^{-1}(x)= \cup_{\alpha} \{x_{\alpha}\}$ for some $x\in K$.\\

Now, Lemma \ref{unifcont} gives us that the union is disoint and that $f|_{K_{\alpha}}:K_{\alpha}\to K$ is bijective.  Moreover, each $K_{\alpha}$ is isolated in $f^{-1}(K)$;
in fact, $d(K_{\alpha}, K_{\beta})>0$ if $\alpha\neq \beta$, otherwise $x_{\alpha}$ and $x_{\beta}$ would be chain-connected.  This implies, as $f$ is a local homeomorphism, 
that
$f|_{K_{\alpha}}:K_{\alpha}\to K$ is a homeomorphism, as we wanted.\\

\end{proof}

Note that as a consequence we obtain a slightly stronger version of Theorem \ref{t1} stated in the introduction (we are not assuming  $X$ to be compact).\\

We also obtain that chain-simply connected continua do not admit non-trivial connected uniformly continuous coverings:\\

\begin{clly}\label{self} Let $K$ be a chain-simply connected continuum, $f:X\to K$ a uniformly continuous discrete metric covering. 
Then, $X=\cup_{\alpha} K_{\alpha}$, a disjoint union of continua and $f:K_{\alpha}\to K$ is a homeomorphism.  
In particular, if $X$ is connected, then $f$ is a homeomorphism.\\
 
\end{clly}

\subsection{Simple connectivity on euclidean spaces.}

For the following results, we will assume that $Y \subset \mathbb{R}^{m}$ and that for every compact set $K \subset Y$ and for every $\rho > 0$, there exists an open, 
simply connected set $U \subset \mathbb{R}^{m}$ such that\\

$$K \subset U \subset K_{\rho} = \{ z \in \mathbb{R}^{m} : d(z, K) < \rho \}.$$\\

Note that if $\varepsilon_0 = d(K, \partial U)$, then for all $x, y \in K$ with $d(x, y) < \varepsilon_0$, it holds that the line segment
$[x, y] = \{tx + (1-t)y \mid t \in [0,1]\} \subset U$.\\

\begin{prop}\label{prop111}
Let \( Y \) be as in the hypotheses above and let \( K \) be a continuum with \( K \subset Y \). Then \( K \) is chain-simply connected. \\

  \end{prop}

\begin{proof}
We need to prove that  for every \( \rho > 0 \) there exists \( \varepsilon_0 > 0 \), \( \varepsilon_0 < \rho \), such that for every \( \varepsilon < \varepsilon_0 \) and for
every \( n \in \mathbb{N} \), any closed \( \varepsilon -n \) chain in \( K \) has a refinement that is \( \rho \)-homotopic to the constant chain.\\

Given $\rho > 0$, consider $\rho/3$ and let $U \subset \mathbb{R}^{m}$ be a simply connected set such that $K \subset U$ and $K \subset U \subset K_{\rho/3}$.
Let  $\varepsilon_0=d(K, \partial U)$  and let $\mathcal{A} = \{x_i\}_{i=1}^{n}$ be a closed $\varepsilon$-$n$-chain in $K$ with $\varepsilon <\varepsilon_0$.\\

As $d(x_i,x_{i+1})<\varepsilon <\varepsilon_0 $, the line segment \([x_i,x_{i+1}] = \{tx_i + (1-t)x_{i+1} \mid t \in [0,1]\} \subset U\).\\

 Let $\alpha_0: [0,1] \to U$ be a closed curve such that there exist $t^{'}_1, \ldots, t^{'}_n \in [0,1]$ with $\alpha_0(t^{'}_i) = x_i$ and 
 $\alpha_{0} |_{_{[t^{'}_i, t^{'}_{i+1}]}} =[x_i, x_{i+1}]$.
 Note that for all $t\in [0,1]$ there exists  $x_i$ such that $d(\alpha_0(t),x_i)<\rho/3 $.
 Let $\alpha_1: [0,1] \to U$ be such that $\alpha_1(t) \equiv x_1$.\\

Since $U$ is simply connected, there exists a homotopy $H: [0,1] \times [0,1] \to U$ of the curves $\alpha_0$ and $\alpha_1$, that is, 
$H(0,t) = \alpha_0(t)$ and $H(1,t) = \alpha_1(t)$. As $d_H(U, K) < \rho/3$, for each $s \in [0,1]$ and $t \in [0,1]$, there exists a point $y\in K$ 
such that $d(H(s,t), y) < \rho/3$. Therefore, by the compactness of $[0,1] \times [0,1]$, there exist two partitions $P = \{s_1, \ldots, s_m\}$ and $Q = \{t_1, \ldots, t_m\}$ of 
the interval $[0,1]$ and points $y_i^j \in K$ that satisfy the following properties:\\

\begin{enumerate}
\item $\text{dist}(H([s_j, s_{j+1}] \times [t_i, t_{i+1}]), y_i^j) < \rho/3$ for all $i,j \in\{ 1, \ldots, m-1\}$,
\item $d_H(H([s_j, s_{j+1}] \times [t_i, t_{i+1}]), H([s_{j+1}, s_{j+2}] \times [t_i, t_{i+1}]) ) < \rho/3$ for all $i,j$,
\item  $d_H(H([s_j, s_{j+1}] \times [t_i, t_{i+1}]), H([s_{j}, s_{j}] \times [t_{i+1}, t_{i+2}]) ) < \rho/3$ for all $i,j$,
\item $y_i^1\in \{x_1, \ldots, x_n\}$ for all $i=1,..., m$ (because $H(0,t) = \alpha_0(t)$ and for all $t\in [0,1]$ there exists  $x_i$ such that $d(\alpha_0(t),x_i)<\rho/3 $ ),
\item $y_i^m = x$ for all $i = 1, \ldots, m$ (because $H(1,t) = \alpha_1(t) \equiv x$).
\item  $y_1^{j}=y_m^{j}=x$   (because $H(s,0)=H(s,1)=x$)\\
\end{enumerate}

For each \( j \), let \( \mathcal{C}_j = \{y_i^{j}\}_{i=1}^{m} \). For all \( j \in \{1, ..., m\} \), by items (1), (2), (3), and (6), it holds that \( \mathcal{C}_j \) is a 
closed \( \rho \)-\( m \) chain. Note that the chain \( \mathcal{A}^{'} := \mathcal{C}_1 \) is a refinement of the chain \( \mathcal{A} \). Furthermore, by items (1) and (2), it
holds that \( d(\mathcal{C}_j, \mathcal{C}_{j+1}) < \rho \). By item (4), the chain \( \mathcal{C}_m \) is constant, which implies that the chain
\( \mathcal{A}^{'} \) is \( \rho \)-homotopic to the constant chain.\\

\end{proof}

  \begin{rk}\label{rk3} In particular, if $Y$ is a non-separating plane continumm, then $Y$ is chain-simply-connected.\\
 
\end{rk}

\begin{rk}\label{rk2}
Let \( Y \) be a non-separating continuum  in \( \mathbb{R}^2 \) with no interior. Then, for each continuum \( K \subset Y \):\\

\begin{enumerate}
    \item \( K \) does not separate the plane because \( Y \) has empty interior.\\
    
    \item Considering \( K \subset S^2 = \mathbb{R}^2 \cup \{\infty\} \), it follows that \( K^c \) is simply connected. Using the
    Riemann Mapping Theorem, given \( \rho > 0 \), there exists an open, simply connected set \( U \) with \( K \subset U \) and \( d_H(U, K) < \rho \), where \( d_H \) is
    the Hausdorff distance. Furthermore, \( U \) can be constructed such that \( \partial U \) is a Jordan curve.\\
    
    \item If \( d(K, \partial U) = \varepsilon_0 > 0 \), then for all \( x, y \in K \) with \( d(x, y) < \varepsilon_0 \), it holds that the line segment
    \([x, y] = \{tx + (1-t)y \mid t \in [0,1]\} \subset U\).\\
\end{enumerate}
\end{rk}

As a consequence of Proposition~\ref{prop111} and Remark~\ref{rk2}, we have the following result:\\

\begin{prop}\label{prop00}
Let \( Y \) be a continuum in \( \mathbb{R}^2 \) that has no interior and does not separate the plane, and let \( K \) be a subcontinuum of \( Y \). Then 
\( K \) is chain-simply connected.\\
  \end{prop}

  \begin{prop}
\label{plane}
A plane continuum is chain-simply-connected if and only if it does not separate the plane.\\
\end{prop}
\begin{proof}
If $K$ does not separates the plane then by Remark \ref{rk3} $K$ is chain-simply-connected, so it remains to prove the other direction.\\

Assume that $K$  is chain-simply-connected and suppose for a contradiction that it is separating. This is equivalent to the existence of a closed annulus $A$ containing $K$ in its
interior and such that $K$ is essential in 
$A$, meaning that $K$ 
separates the two components of the boundary of $A$.  Let $f$ be any degree two covering map of $A$. Now, by Theorem \ref{t1}, $f^{-1}(K)=K_1 \cup K_2$, and $f|_{K_i}:K_i\to K$
is a homeomorphism. Note that as $K_1\cap K_2=\emptyset$ we can find  an essential annulus $U$  containing $K$ such that $f^{-1}(U)$ contains $K_1$ and $f^{-1}(U)\cap K_2=\emptyset$. 
Now, the preimage of any
essential annulus is an essential annulus such that the restriction of $f$ to this annulus is $2:1$.  This contradicts the fact that $f^{-1}(U)\cap K_2=\emptyset$.\\

%

\end{proof}

\section{Examples}\label{4}
In this section we review known examples of indecomposable continua supporting self-coverings, together with contrasting cases where none exist beyond the trivial ones; see \cite{timm-survey} for a broader survey of constructions and obstructions to self-covers in general. As explained in the introduction, understanding the dynamics of coverings and branched coverings on indecomposable continua is key to solving several well known open problems.\\

{\bf I. Non-separating planar continua.} Note that it follows from Corollary \ref{self} and Proposition \ref{plane} that non-separating planar continua such as the pseudo-arc or the Knaster continuum do not admit self-coverings other than homeomorphisms. However, in the last section we will show that they both admit branched coverings with critical points. In fact, Corollary \ref{self} rules out nontrivial coverings of such continua $K$ by \emph{any} continuum $X$, not merely self-coverings. Note that exactly 2-to-1 maps from continua onto the Knaster continuum do exist, as shown by Dębski, Heath and Mioduszewski \cite{dhm}; by the above, such maps cannot be coverings in our sense. This is consistent with the general theory of exactly $k$-to-1 maps, which need not be coverings even in comparatively tame settings; see \cite{heath-survey} for a survey.\\

{\bf II. The pseudo-circle.} It is known that for any integer $k>1$, the pseudo-circle is a $k$-fold covering space of itself --- the case $k=2$ is due to Heath \cite{heath}, generalized to all $k$ by Bellamy \cite{bellamy}. However, the following remains unknown: does there exist a plane branched covering with a totally invariant pseudo-circle?\\

{\bf III. The derived from Anosov endomorphism continuum.} The first example of a sphere branched covering supporting a totally invariant indecomposable continuum was given in \cite{plig}. Whether or not there exists a rational function with an indecomposable continuum as its Julia set is a well-known unsolved problem.\\

{\bf IV. The solenoid.} It is known that the dyadic solenoid admits $d:1$ self-coverings for all odd integers $d$ (see \cite{chino} and \cite{sole}; see also \cite{timm-survey} for this and related constructions). As we will use this fact to give an example of a branched covering of the Knaster continuum with a peculiar critical portrait, we give a proof of it in the next section.\\

{\bf V. Pseudosolenoids.} Unlike the solenoid, pseudosolenoids are hereditarily indecomposable. Using Fox's theory of overlays, Boroński and Sturm proved that a $P$-adic pseudosolenoid $S$ admits a $d$-fold covering onto itself if and only if $d$ is relatively prime to all but finitely many primes in $P$, and that any connected finite-sheeted covering space of $S$ is homeomorphic to $S$ \cite{boronski-sturm}.\\

\begin{rk} It follows that the pseudo-circle, the derived from Anosov endomorphism continuum, the solenoid, and those pseudosolenoids admitting a nontrivial self-covering are not
chain-simply-connected.
\end{rk}

\section{covering maps of solenoids}\label{5}

If $f:X\to X$, we denote $\varprojlim (X,f)$ the inverse limit space of $f$, that is, the space of sequences $\{x_n\}_{n\in\N}$ such that 
$f(x_n)=x_{n-1}$ for every $n>0$.\\

If $g:X\to Y$ verifies $gf=hg$ for some $h:Y\to Y$, then $g$ induces a map
$\hat g:\varprojlim (X,f)\to \varprojlim (Y,h)$, where $\hat g(x_0, x_1, \ldots, x_n, \ldots) = (g(x_0), g(x_1), \ldots, g(x_n), \ldots) $.\\

The diadic solenoid $S_2$ is the inverse limit set of the map $m_2:S^1\to S^1$ given by $m_2(z)=z^2$; also, for every $d\in \N$, let $m_d: S^1\to S^1$ be $m_d(z)=z^d$. \\

\begin{thm}\label{solenoide}  For every odd natural number $d$, the map $\hat m_d:S_2\to S_2$ is a degree $d$ self covering.\\

If $d=2^k(2l+1)$, then $\hat m_d:S_2\to S_2$ is a degree $2l+1$ self-covering. \\
\end{thm}

This result is already contained in \cite{sole}.  We give a simple proof here for the sake of a self-contained exposition.\\

\begin{proof} 
We first note that if $d$ is an odd integer, then $m_d(x)=m_d(y)$ only if $m_2(x)=m_2(y)$. \\
%

It is clear that if $d=2^k$, then $\hat m_d:S_2\to S_2$
is a homeomorphism.  So, we just have to prove the assertion for $d$ odd.\\

Now, let $d$ be an odd integer and ${\bf x}=(x_0, x_1, \ldots, x_n, \ldots)\in S_2$.  We claim that ${\bf x}$ has exactly $d$ preimages by $\hat m_d$. 
Let $x_0^1, \ldots, x_0^d$ be the $d$ preimages of $x_0$ by $m_d$.  Note that for each $j\in \{1, \ldots, d\}$, 
there is exaclty one point $x_1^j$ such that $(x_1^j)^d=x_1$ and $(x_1^j)^2=x_{0}^j$. We finish the proof inductively.\\

%
%
\end{proof}

%
%
%

%
%
%
%

\section{branched coverings}\label{6}
 
 In this section we define metric branched coverings and prove Theorem \ref{t2}.\\

\begin{defi}
\label{bc} 
Let $X$ be a metric space. We say that a continuous map $f:X \to Y$ is a {\em branched metric covering}, if for every point $y\in Y$ and any positive number $\epsilon$ 
there exists a neighborhood $U$ of $y$ such that:\\

\begin{enumerate}
\item
$f^{-1}(U) =\cup_{i\in I} U_i$ is a disjoint union of open sets $U_i$ each one of diameter less than $\epsilon$, 
\item 
each $U_i$ contains precisely one point $x_i\in f^{-1}(y)$, and
\item
$f|_{U_i\setminus \{x_i\}}:U_i\setminus \{x_i\}\to U\setminus \{y\}$ is a degree $r_i\geq 1$ metric covering. \\
\end{enumerate} 
Such a neighborhood $U$ of $y$ is called $\epsilon$-admissible and the $U_i's$ are called the parts of $f^{-1}(U)$. 
\end{defi}

The number $r_i$ is called the local degree of $f$ at $x_i$ denoted $\deg(f,x_i)$. Note that at points $x_i\in U_i$ where $\deg(f,x_i)=1$, 
$f|_{U_i}:U_i\to U$ is a homeomorphism.
Note also that the definition allows different local degrees at points in the same fiber $f^{-1}(y)$.
A critical point is a point $p\in X$ with $\deg(f,p)\geq 2$.  A critical value is a point $y\in Y$ such that $f^{-1}(y)$ contains a critical point.  
Note that the set of critical points
is discrete in $X$.  The set of critical points is denoted $C_f$. If $c\in C_f$, then the multiplicity of
$c$ is $m(c):=\deg(f,c)-1$.
\\

%
%

We say that the branched metric covering $f:X\to Y$ has degree $d$ if $\# f^{-1}(y)=d$ for all $y\in Y \setminus f(C_f)$. Note that in this case
$$
\sum_{x\in f^{-1}(y)} \deg(f,x)=d
$$
holds for every $y\in Y$.\\

%
%
%
%

In what follows, we will assume that $X$ and $Y$ are compact connected metric spaces and $f:X\to Y$ is a degree $d$ branched metric covering.\\

%
%
%
%
%

For the proof of Theorem \ref{t2} we introduce the concept of conglomerate:  \\

\begin{defi}
\label{conglo}
Let $K$ be compact and $\mathcal V$ an open cover of $K$. Say that $V,W\in \mathcal V$ are related if there exist $n\in\N$ and a sequence
$V_0,\ldots,V_n$ in $\mathcal V$ such that $V_0=V$, $V_n=W$ and $V_i\cap V_{i+1}\neq\emptyset$ 
for every $0\leq i<n$. The equivalence classes of this equivalence relation are called $K$-$\mathcal V$-conglomerates.\\
\end{defi}

If $\mathcal U$ is a finite set of open sets, define $mesh(\mathcal U)$ as the maximum of the diameters of the elements of $\mathcal U$.
For a conglomerate $\mathcal V'$ define $\cup \mathcal V'=\cup\{V:V\in \mathcal V'\}$. As the space is not assumed to be locally connected, 
we cannot expect $\cup \mathcal V'$ to be connected. However, if 
$K$ is connected and $\mathcal V$ is a cover of $K$ such that $K\cap V\neq\emptyset$ for every $V\in \mathcal V$, then $\mathcal V$ is a 
$K$-$\mathcal V$-conglomerate. Also, if $K$ is not connected and $mesh( \mathcal V)$ is small, then $ \mathcal V$ is not a conglomerate. 
A refinement of a cover $\mathcal V$ of $K$ is a cover $\mathcal V'$ of $K$ such that every element of $\mathcal V'$ is contained in an element of
$\mathcal V$.\\

With all this in mind, the following can be easily verified:\\

\begin{lemma} 
\label{conexos}
Let $\mathcal V_n$ be a sequence of open covers of $K$ such that for every $n$, $\mathcal V_n$ is a refinement of $\mathcal V_{n-1}$ and
$mesh (\mathcal V_n)$ converges to zero. Assume also that for every $n$, $\mathcal V'_n$ is a $K$-$\mathcal V$-conglomerate 
and that $\cup \mathcal V'_n$ is a decreasing sequence.\\
Then $\cap_n\cup\mathcal V_n'$ is a component of $K$. \\
\end{lemma}

\noindent
{\bf Proof of Theorem \ref{t2}.}\\
Let $\varepsilon>0$ be given and let $\mathcal U$ be a cover of $K$ made by a finite number of $\varepsilon$-admissible neighborhoods of points of $K$.
Define $\mathcal V$ as the collection of the parts of $\mathcal U$ (see definition \ref{bc}).\\
It is claimed that for every $f^{-1}(K)$-$\mathcal V$-conglomerate $\mathcal V'$, it holds that
$$
f(\cup \mathcal V')\supset K
$$
Indeed, let $x\in f^{-1}(K)\cap (\cup\mathcal V')$, and take any point $y$ in $K$. It must be shown that $y$ has a preimage in $\cup V'$.  
As $K$ is connected, there exists a sequence $U_0,\ldots, U_n$ of elements of $\mathcal U$ such that $U_i\cap U_{i+1}\neq\emptyset$, $f(x)\in U_0$ and 
$y\in U_n$. Let $V_0$ be the part of $f^{-1}(U_0)$ that contains $x$. Note that $V_0\in\mathcal V'$. As $f|_{V_0}$ is onto $U_0$, 
it follows that there exists a part $V_1$ of $f^{-1}(U_1)$ intersecting $V_0$. Obviously $V_1\in\mathcal V'$. By an induction argument it 
can be seen that there exists a 
sequence $V_0,\ldots,V_n\in \mathcal V$ such that $f(V_i)=U_i$ and $V_i\in\mathcal V'$ for every $0\le i\leq n$. Then $y$ contains a preimage in 
$V_n\subset\cup\mathcal V'$. This proves the claim.\\
All this was made for a given $\varepsilon>0$. Now consider a decreasing sequence $\varepsilon_n\to 0$, and for each $n$, let $\mathcal U_n$ be a cover of $K$ made by 
$\varepsilon_n$-admissible neighborhoods of points of $K$ such that $\mathcal U_n$ is a refinement of $\mathcal U_{n-1}$. Let $\mathcal V_n$ be 
the cover of $f^{-1}(K)$ made by the parts of $f^{-1}(U)$ with $U\in\mathcal U_n$. The claim above implies that the number of   
$f^{-1}(K)$-$\mathcal V_n$-conglomerates is at most the degree of $f$. Let $K'$ be a component of $f^{-1}(K)$; note that for each $n$ there is a conglomerate $\mathcal V'_n$ of $\mathcal V_n$ such that $K'\subset \cup\mathcal V'_n$. This, together with Lemma \ref{conexos}, implies that $f(K')=K$.
The fact that $f|_{K'}$ is a branched covering onto $K$ is now immediate.\\

%
%
%

\section{Riemann-Hurwitz}\label{7}
The Riemann-Hurwitz formula establishes a relation between the degree and the number of critical points of branched coverings $f:X\to Y$. 
This relation involves the Euler characteristic of the spaces $X$ and $Y$. A priori, it has no sense when the spaces are not locally connected.\\
  
We introduce some notations:\\
\begin{itemize}
\item
$BC(X,Y)$ (resp. $BC(X)$) is the space of branched coverings from $X$ to $Y$ (resp. from $X$ to $X$).
\item
$BC^*(X,Y)$ (resp. $BC^*(X)$) is the space of branched coverings from $X$ to $Y$ (resp. from $X$ to $X$) having degree $d>1$.
\item
$\deg(f)$ is the degree of $f$, and $r(f) = \sum_p (k_p - 1)$ is the sum of the multiplicities of the critical points of $f$, where $k_p$ denotes the local degree of $f$ at the critical point $p$.
\item 
The Euler characteristic of an open connected subset of the sphere $S^2$ is defined as $\chi(X)=2-n(X)$, where $n(X)$ is the number of connected components of $S^2\setminus X$.\\
\end{itemize}

The classical Riemann-Hurwitz formula holds for branched coverings between finite cell complexes and also between open connected subsets of the sphere.
For every $f\in BC(X,Y)$:\\

\begin{equation}
\label{e1}
\chi(X)= \deg (f)\chi(Y)-r(f)\\
\end{equation}

\begin{rk}
A well known consequence of this formula is the following: if $f\in BC^*(S^2)$ and $X$ is a completely invariant continuum ($f^{-1}(X)=X$), then $S^2\setminus X$ has $0$, $1$, $2$,
or infinitely many components. In particular, no Wada lake with a finite number $L>2$ of lakes can be a completely invariant subcontinuum of some $f\in BC^*(S^2)$. This rules out,
for instance, the Plykin attractor, which has exactly four complementary domains but arises as the attractor of a homeomorphism rather than of a branched covering of degree greater 
than one. By contrast, a Wada lake with infinitely many lakes that \emph{is} a completely invariant subcontinuum of some $f\in BC^*(S^2)$ was constructed in \cite{plig}.\\
\end{rk}

Going further, formula (\ref{e1}) can itself be used to \emph{define} the Euler characteristic of a continuum $X$, even when $X$ is not a cell complex and the classical definition
does not apply.\\

\begin{defi}
Let $X$ be a continuum with $BC^*(X)\neq \emptyset$. If the function
$$
f\in BC^*(X)\longmapsto \frac{r(f)}{\deg(f)-1}
$$
is constant, we say that $X$ has a \emph{well-defined Euler characteristic}, equal to this common value:
$$
\chi(X)=\frac{r(f)}{\deg(f)-1}, \qquad f\in BC^*(X).
$$\\
\end{defi}

For instance, this gives $\chi(X)=0$ for the solenoid; and, as we show next, $\chi(K)=1/2$ for the Knaster continuum $K$.  It is worth stressing that this notion is genuinely new, 
not merely an extension of an existing one to a wider domain. The standard shape-theoretic approach to defining topological invariants for arbitrary compacta, via Čech cohomology, 
does apply to every continuum -- but it is blind to the phenomenon we are after here. Chainable continua, such as the Knaster continuum, have trivial Čech cohomology in positive 
degrees; the corresponding shape-theoretic Euler characteristic would therefore assign $\chi(K)=1$ to the Knaster continuum, exactly as
it would to a single point, revealing nothing of the finer structure captured by our approach. Our definition, by contrast, extracts a non-trivial rational invariant, $\chi(K)=1/2$, directly 
from the combinatorics of its branched self-coverings -- information that is simply invisible to the classical shape-theoretic 
invariant.\\

\subsection*{The Knaster continuum.}

The Knaster continuum is an indecomposable planar continuum. It can be constructed (as it was by Knaster a hundred years ago) as an intersection of discs, but it is also 
the inverse limit of the second Chebyshev polynomial $T_2$, defined on the interval $I=[-1,1]$ by $T_2(x)=2x^2-1$.
In general, $T_n$ will denote the $n_{th}$ Chebyshev polynomial. That is, $T_n$ is the induced map on $I$ by the function $m_n:S^1\to S^1$, $m_n(z)=z^n$ and the projection
$\pi:S^1\to I$, $\pi (z)=Re(z)$ ($T_n\pi=\pi m_n$).\\

We consider $K$ the inverse limit of the map $T_2$ and define $\hat T_d$ on $K$ as
$$\hat T_d(x_0,x_1,\ldots)=(T_d(x_0),T_d(x_1),\ldots).$$
As $T_d$ and $T_2$ commute it follows that
$\hat T_d$ is a map from $K$ to $K$.  
Moreover, as $\pi m_2 = T_2\pi$, $\pi$ induces a map $\hat \pi: S_2\to K$, where $S_2$ is the diadic solenoid which we are indetifying here
with the inverse limit of the map $m_2$.  Note that $\hat \pi$ is a $2:1$ branched covering from the diadic solenoid to the Knaster continuum with
exactly one branching point, namely $(1,1,\ldots, 1, \ldots)$.\\

\begin{prop}
\label{euler}
 $BC^*(K)$ is not empty. Moreover, if $f\in BC^*(K)$, then $2r(f)+1=\deg(f)$.\\
\end{prop}

\begin{proof} 
Claim 1. For every $d$, $\hat T_d:K\to K$ is a branched covering with one critical value $\hat \pi (1,1,\ldots, 1, \ldots)$.  
If $d$ is an odd integer, then $\hat{T}_d$ is $d:1$ with $(d-1)/2$ critical points of multiplicity $1$.  If $d=2^k(2l+1)$, then $\hat{T}_d$
is $2l+1: 1$ with $2l$ critical points with multiplicity $1$. \\

To prove the claim, note that $\hat{T_{2^k}}$ is a homeomorphism as it is the natural action on $K$ induced by the bonding map $T_2$.  
So, we only have to prove the assertion for $d$ odd.
By Theorem \ref{solenoide}, if $d$ is odd, then $\hat{m}_d:S_2\to S_2$ is a $d:1$ self-covering.  As $\hat \pi$ is $2:1$ except at $(1,1,\ldots, 1, \ldots)\in S_2$, it follows that
$\hat T_d$ is $d:1$ except at $\hat T_d^{-1}((1,1,\ldots, 1, \ldots))$.  Note that $\hat T_d^{-1}((1,1,\ldots, 1, \ldots))$ consists exaclty of $(d-1)/2$ points that are the $\hat \pi$
projection of the
$d-1$ $d_{th}$ roots of unity different from $1$.  Moreover, at each of these $(d-1)/2$ points $\hat T_d$ is $2:1$, and hence the multiplicity of each of them is $1$.\\

Claim 2. Let $f:K\to K$ be a degree $d$ branched covering, $d\geq 1$.  Then, $f$ has exactly one critical value $p$: the endpoint of $K$. Moreover, $p$ is fixed and regular, and
any other preimage of $p$ is a critical point of multiplicity $1$.  In particular, $d$ is odd and  there are exaclty $\frac{d-1}{2}$ critical points. The second assertion of 
the proposition follows.\\

To prove this claim, note that any point other than $p$, $K$ is locally homeomorphic to $C\times I$, where $C$ is the Cantor set and $I=(0,1)$. 
This implies that $p$ must be fixed and regular.  Any other preimage of $p$ is a point with local structure $C\times I$ which is mapped to the 
endpoint of $K$.  Then, necesarilly the point is critical and its multiplicity is $1$ (as composant are sent in composants, 
the critical points are critical points of interval maps).
 
\end{proof}

\begin{rk}
\label{c4}
The Euler characteristic of the Knaster continuum is $\chi(K)=1/2$.\\

\end{rk}

\begin{prop}  Let $f:S^2\to S^2$ be a branched covering of degree $d\geq 1$. Let $X\subset S^2$ be a non-separating continuum such that $f^{-1}(X)=X$. 
Then $X$ contains $d-1$ critical points counted with multiplicity.
\end{prop}
 
\begin{proof} 
It is obviuos that if $X$ is completely invariant then so is $Y=S^2\setminus X$. As $X$ does no separate the sphere, $Y$ is connected, and as $X$ is connected, $\chi(Y)=1$. On the 
other hand, as $\chi(S^2)=2$ every $f$ as in the hypothesis of the proposition has $2d-2$ critical points in the sphere 
and $d-1$ in $Y$. It follows that $X$ contains the remaining $d-1$ critical points of $f$ and the assertion follows.
\end{proof}
%
%
%

\begin{rk}  The previous result shows that the Knaster continuum cannot be completely invariant for a branched covering of the sphere of degree $d>1$.
\end{rk}

In a sequel of this article we  define a class of spaces  that are not necessarily locally connected but have a well defined 
Euler characteristic if they admit self-branched coverings.  Indeed we  prove:\\

\noindent
{\em Let $X$ be a continuum containing no indecomposable subcontinua and such that 
the intersection of two connected subspaces of $X$ is connected.  If $BC^*(X)\neq \emptyset$, then $\chi(X)=1$ .}\\

This result is then applied to study dynamics of branched coverings of $X$.\\

We finish with a simple lemma that, in some particular cases,  extends the Riemann-Hurwitz formula to general spaces having a well defined Euler characteristic.

\begin{lemma} Let $X$ and $Y$ be continua with well defined Euler characteristic and let $f:X\to Y$ be a branched covering. Suppose, furthermore that there exists branched covering
$g:X\to X$, $h:Y\to Y$ such that:

\[
\begin{tikzcd}
X \arrow[r,"g"] \arrow[d,"f"'] &
X \arrow[d,"f"] \\
Y \arrow[r,"h"'] &
Y\\
\end{tikzcd}
\]

Then, the Riemann-Hurwitz formula holds for $f$:\\

$$\chi(X)=\deg(f)\chi(Y) - r(f) .$$\\
 
\end{lemma}

\begin{proof}  First note that the following equations hold:\\

$$r(fg)=\deg(g)r(f)+r(g),$$
$$r(hf)= \deg(f)r(h)+r(f) .$$\\

Now, as $fg=hf$, it follows that:\\

$$0=r(f)(\deg(g)-1)+r(g)-\deg(f)r(h).$$\\

Now, we divide by $\deg(g)-1$, note that $\deg(g)=\deg(h)$ and use the definition of Euler characteristic to get:

$$0=r(f)+\chi(X)-\deg(f)\chi(Y),$$\\

exactly as we wanted.

\end{proof}

As an application, note that the branched covering $f:S_2\to K$ described above, from the diadic solenoid to the Knaster continuum is $2:1$ with exactly one critical point of 
multiplicity $1$.  We see that:

$$0=\chi(S_2)=\deg(f)\chi(K)-r(f)= 2(1/2)-1.$$\\

\noindent
{\bf Acknowledgements.}  We are indebted to Professor Jernej \v{C}in\v{c}, whose insightful guidance on indecomposable continua greatly benefited this work.\\

\end{document}